\documentclass[a4paper,11pt,DIV=11,oneside,headings=small,bibliography=totoc,abstract=on]{scrartcl}
\usepackage{amssymb,amsmath,amsfonts,amsthm}
\usepackage[markup=default]{changes}
\usepackage[english]{babel}
\usepackage[utf8]{inputenc}
\usepackage{enumerate,tikz}

\usepackage[noblocks]{authblk}

\renewcommand{\epsilon}{\varepsilon}

\newcommand{\RR}{\mathbb{R}}
\newcommand{\CC}{\mathbb{C}}
\newcommand{\NN}{\mathbb{N}}    
\newcommand{\ZZ}{\mathbb{Z}}

\newcommand{\TT}{\mathbb{T}}

\newcommand{\ran}{\operatorname{Ran}}

\newtheorem{theorem}{Theorem}
\newtheorem*{lemma}{Lemma}

\newtheorem*{question}{Question}
\theoremstyle{definition}

\theoremstyle{remark}

\usepackage{microtype}
\usepackage{ellipsis}
\usepackage[textsize=tiny,shadow]{todonotes}
\begin{document}
%
%
%
%
%
%
%
%
\title{Laplace-eigenfunctions on the torus with high vanishing order}
\author{Matthias T\"aufer}
\affil{Technische Universit\"at Dortmund, Fakult\"at f\"ur Mathematik, Germany}
\date{\vspace{-2em}}
\maketitle
\begin{abstract}
 We use the sum-of-squares theorem from number theory to construct eigenfunctions of the Laplacian on the $d$-dimensional torus, $d \geq 2$, which vanish to any prescribed order at some point. 
 These functions are then applied to provide a negative answer (in dimension $d \geq 2$) to a question in the context of quantitative unique continuation for spectral projectors of Schr\"odinger operators, asked by Egidi and Veseli\'c in~\cite{EgidiV-16}.
\end{abstract}


For a compact, connected manifold $X$ with $C^\infty$ Riemannian metric, the Laplace (or Laplace-Beltrami) operator $\Delta$ on $X$ is a non-positive, self-adjoint operator.
One says that $f: X \to \CC$ vanishes to order $N > 0$ at $x_0 \in X$ if
\[
 \limsup_{\delta \to 0} \frac{\sup_{y \in B_\delta(x_0)} \lvert f(x) \rvert}{\delta^N} < \infty.
\]
In~\cite{DonnellyF-88}, Donnelly and Feffermann proved that if $f$ is a nonzero eigenfunction of $-\Delta$ on $X$ with eigenvalue $\lambda$ it will vanish to order \emph{at most} $c \sqrt{\lambda}$ at any point $x \in X$ where $c$ is a constant that depends only on $X$.
%

A complementary question is: Given any vanishing order $N$, can we find eigenfunctions vanishing to order \emph{at least} $N$ at some point $x_0$?
In dimension $d = 1$, for instance if $X$ is an interval with its end points identified, this is not possible as soon as $N > 1$. 
This follows immediately from the fact that all nonzero eigenfunctions are of the form
$\alpha \sin (\omega x + y_0) + \beta \cos(\omega x + y_0)$ and vanish to order at most $1$ at any point.
However, in dimension $d \geq 2$ it is indeed possible to find eigenfunctions of high vanishing order as our first result, Theorem~\ref{thm:1}, shows.
Let us denote by $\TT^d = [- \pi, \pi)^d$ the $d$-dimensional torus, i.e. opposite sides of $[-\pi,\pi)^d$ are identified, and by $\Delta$ the Laplace operator on $\TT^d$.
\begin{theorem}
 \label{thm:1}
 Let $d \geq 2$. 
 For every $N > 0$ there is a nonzero function $f \in L^2(\TT^d)$ with $- \Delta f = \lambda f$ for some $\lambda > 0$ such that $f$ vanishes to order at least $N$ at $0$.
\end{theorem}

 By translation invariance on $\TT^d$, we may replace $0$ by any point $x_0 \in \TT^d$.

\begin{proof}
 For $\lambda \geq 0$ let 
 $
  I_\lambda
  :=
  \{
   k \in \ZZ^d \colon \lvert k \rvert^2 = \lambda
  \}
 $
 and for $k \in \ZZ^d$ let $\psi_k(x) := \exp (i k \cdot x) \in L^2(\TT^d)$.
%
 A function $f \in L^2(\TT^d)$ is an eigenfunction to the eigenvalue $\lambda \geq 0$ if and only if $f$ is of the form
 \[
  f = \sum_{k \in I_\lambda} \mu_k \psi_k, \quad \mu_k \in \CC.
 \]
 We expand the $\psi_k$ in a Taylor series around $0$
 \[
 \psi_k(x) 
 =
 \sum_{\alpha \in \NN_0^d} \frac{1}{\alpha!} (D^\alpha \psi_k)(0) \cdot x^\alpha
 =
 \sum_{\alpha \in \NN_0^d} \frac{(i k)^\alpha}{\alpha!} x^\alpha
\]
where we used multindex notation, i.e. given $\alpha = (\alpha_1, ..., \alpha_d) \in \NN_0^d$, we write
 $\alpha!  := \alpha_1 ! \cdot \dots \cdot \alpha_d!$, 
 $D^\alpha := \partial_{x_1}^{\alpha_1} \dots \partial_{x_d}^{\alpha_d}$,
 $k^\alpha := k_1^{\alpha_1} \cdot \dots \cdot k_d^{\alpha_d}$,
 $\lvert \alpha \rvert_1 := \alpha_1 + \dots + \alpha_d$.
Then, $f$ can be expressed as
\[
 f(x)
 =
 \sum_{k \in I_\lambda} \mu_k \psi_k
 =
 \sum_{\alpha \in \NN_0^d} \frac{(i x)^\alpha}{\alpha!} 
  \left(
   \sum_{k \in I_\lambda} \mu_k k^\alpha
  \right).
\]
Since the Taylor series is locally absolutely convergent, $f$ vanishes to order $N$ at $0$ if for all $\alpha \in \NN_0^d$ with $\lvert \alpha \rvert_1 \leq N$, we have
\[
 \sum_{k \in I_\lambda} \mu_k k^\alpha = 0.
\]
This is a system of finitely many linear equations, indexed by $\alpha$, with variables $\{ \mu_k \}_{k \in I_\lambda}$.
More precisely, we have
\begin{align*}
 \sharp \{ \text{equations} \} 
 &= 
 \sharp \{ \alpha \in \NN_0^d \colon \lvert \alpha \rvert_1 \leq N \}  =: C(N),
 \quad \text{\emph{fixed}, once we chose $N$},\\
 \sharp \{ \text{variables} \} 
 &=
 \sharp
 \{ k \in \ZZ^d \colon \lvert k \rvert^2 = \lambda \}.
\end{align*}
If $\sharp \{ \text{variables} \} > \sharp \{ \text{equations} \}$ then there will be a non-trivial solution $\{ \mu_k \}_{k \in I_\lambda}$. 
This will yield a function $f$ which vanishes to order $N$ at $0$.
Since $f$ is a nontrivial linear combination of orthogonal $\psi_k$, it is non-zero in $L^2(\TT^d)$ sense.

Thus, it remains to show that for every $C \in \NN$, there is $\lambda \geq 0$ such that $ \sharp
 \{ k \in \ZZ^d \colon \lvert k \rvert^2 = \lambda \} \geq C$.
Clearly, it suffices to establish this in dimension $d = 2$ and in this case, the task boils down to finding $\lambda \geq 0$ such that the number of all pairs $(x,y) \in \ZZ^2$ satisfying $x^2 + y^2 = \lambda$ exceeds $C$.
For $\lambda \in \NN$, this number is explicitly given by the so-called sum-of-squares theorem, sometimes also referred to as Gauss's formula, which can be found in~\cite{Gauss-1801}. 
See also \cite[Chapter 1]{Fricker-82} for a more modern reference.
The sum-of-squares theorem states that for $\lambda \in \NN$ with prime factor decomposition
 \[
  \lambda = p_1^{a_1} \cdot \dots \cdot p_k^{a_k} \cdot q_1^{b_1} \cdot \dots \cdot q_l^{b_l} \cdot 2^c,
 \]
 where $p_i$ are primes of the form $4k + 1$ and $q_i$ are primes of the form $4k+3$, the number of pairs $(x,y) \in \ZZ^2$ with $x^2 + y^2 = \lambda$ is
 \[
  \begin{cases}
   4 \cdot (1 + a_1) \cdot \dots \cdot (1 + a_n) 
   \quad
    & \text{if all $b_i$ are even},\\
   0
   & \text{else}.
  \end{cases}
 \]
 Choosing e.g. $\lambda = 5^C$, this implies
 \[
  \sharp \{ (x,y) \in \ZZ^2 \colon x^2 + y^2 = \lambda \} 
  =
  4 \cdot (1 + C)
  \geq 
  C.
  \qedhere
 \]
\end{proof}

An analogous argument works if the Laplacian on the torus is replaced by the Laplacian on a hypercube with Dirichlet or Neumann boundary conditions. 

Our motivation to study the vanishing order of eigenfunctions comes from quantitative unique continuation principles (UCPs) or observability estimates with explicit dependence on the geometry.
The above notion of vanishing to order $N$ at a point $x_0$ can be understood as vanishing with respect to the $\sup$ norm.
The UCPs we are interested in are then bounds on the vanishing order with respect to the $L^2$ norm, i.e. estimates of the form
\begin{equation}
 \label{eq:ucp}
 \int_{B_\delta} \lvert f \rvert^2 \geq \delta^M \int_{\TT^d} \lvert f \rvert^2,
 \quad
 0 < \delta \leq \pi
\end{equation}
where we set $B_\delta := B_\delta(0)$.
The following lemma clarifies the connection between functions of high vanishing order and counterexamples to~\eqref{eq:ucp}.
\begin{lemma}
 \label{lem:sup-L2}
 If $0 \neq f \in L^2(\TT^d)$ vanishes of order $N$ at $0$, then~\eqref{eq:ucp} cannot hold with $M = 2 N$.
\end{lemma}

\begin{proof}
 We estimate by Hoelder's inequality
 \[
 \lim_{\delta \to 0}
  \frac{\int_{B_\delta} \lvert f \rvert^2}{\delta^{2 N}}
  \leq
 \lim_{\delta \to 0}
  \operatorname{Vol}(B_\delta) \cdot \left( \frac{\sup_{B_\delta} \lvert f \rvert}{\delta^N} \right)^2
  =
  0
 \] 
 where we used that the second term on the right hand side remains bounded as $\delta \to 0$.
 Thus, Ineq.~\eqref{eq:ucp} cannot hold for $M = 2 N$.
\end{proof}

Recently, UCPs as in~\eqref{eq:ucp} have been studied extensively and have lead to several applications in the field of random Schr\"odinger operators (Wegner estimates and localization for non-ergodic random Schr\"odinger operators~\cite{Rojas-MolinaV-13,Klein-13}, Wegner estimates for the random breather model~\cite{TaeuferV-15}, Wegner estimates and initial scale estimates for models with non-linear dependence on the randomness~\cite{NakicTTV}) as well as in control theory~\cite{NakicTTV}.
In the process, Ineq.~\eqref{eq:ucp} has been generalized to larger and larger classes of functions.
More precisely, let $E \geq 0$ and $V \in L^\infty(\TT^d)$.
Then it has been proved that Ineq.~\eqref{eq:ucp} holds with a uniform $M$
\begin{enumerate}[i)]
 \item 
 for all $f$ which are eigenfunctions of $-\Delta + V$ with eigenvalue $\lambda \leq E$, cf.~\cite{Rojas-MolinaV-13},
 \item
 for all $f \in \ran \chi_{[\lambda - \epsilon, \lambda]}(- \Delta + V)$, the range of the spectral projector onto the interval $[\lambda - \epsilon, \lambda]$ for $\lambda \leq E$ where $\epsilon > 0$ is a $\delta$-dependent parameter, cf.~\cite{Klein-13},
 \item
 for all $f \in \ran \chi_{(- \infty, E]}(-\Delta + V)$, cf. \cite{NakicTTV-15, NakicTTV},
 \item
 for all $f \in L^2(\TT^d)$ which are (finite or infinite) linear combinations of $-\Delta + V$-eigenfunctions the coefficients of which satisfy a certain $E$-dependent exponential decay condition, cf.~\cite{TaeuferT-17}.
\end{enumerate}
Furthermore, if $V \equiv 0$, Ineq.~\eqref{eq:ucp} has been proved with a uniform $M$
\begin{enumerate}[i)]
 \setcounter{enumi}{4}
 \item
 for all $f \in L^2(\TT^d)$ whose Fourier Transform $\hat f \in \ell^2(\ZZ^d)$ has support on sites contained in a (not necessarily centered) cube of side length $E$, cf.~\cite{EgidiV-16},
 \item
 for all $f \in L^2(\TT^d)$ whose Fourier Transform $\hat f \in \ell^2(\ZZ^d)$ has support on sides contained in a union of a fixed number of cubes of side lengths $E$, cf.~\cite{EgidiV-16}.
\end{enumerate}
Let us emphasize that we only cited very special cases here:
All of the results i) to vi) have actually been stated (and the constant $M$ holds uniformly) in a much more general setting, namely in a multiscale setting where $\TT^d$ is replaced by an infinite family of hypercubes and where $B_\delta$ is replaced by a family of equidistributed arrangements of balls.
Furthermore, in i) to iv), the constant $M$ depends on $V$ only via $\lVert V \rVert_\infty$ whence it is uniform over a large class of potentials $V$.
Even though this constituted a large part of the originality of these results and was crucial for the applications envisioned therein, we omitted these aspects here for the sake of a simpler presentation and refer to~\cite{NakicTTV} for a more comprehensive discussion.

One interesting part about the results v) and vi) is that they allow for linear combinations of eigenfunctions of arbitrarily high eigenvalues - with some restrictions on their Fourier transforms.
This lead Egidi and Veseli\'c to a question, a special case of which we cite here.
\begin{question}[Special case of~{\cite[Question 3]{EgidiV-16}}]
 \label{question:EgidiV-16}
 Let $V \in L^\infty(\TT^d)$ and fix $w \in (0, \infty)$.
 Is there a constant $M$, which may depend on $w$ and $V$, such that for all $E_0 \in \RR$, all $0 < \delta < \pi$, and all $f \in \ran \chi_{[E_0 - w, E_0]}(-\Delta + V)$, 
 the estimate
 \[
  \int_{B_\delta} \lvert f \rvert^2 \geq \delta^M \int_{\TT^d} \lvert f \rvert^2.
 \]
 holds true?
\end{question}
Again, the authors in~\cite{EgidiV-16} asked the question in a more general geometric setting but since this is not relevant for our reasoning, we omit it here.
Furthermore, they suggested that the answer might depend on the dimension and/or regularity properties of $V$. 

\begin{theorem}
 The answer to Question~3 in~\cite{EgidiV-16} is no in dimension $d \geq 2$.
\end{theorem}

\begin{proof}
 It suffices to consider $V = 0$, i.e. the case of the pure Laplacian.
 By the above lemma, it suffices to find for every $N \geq 0$ a function $f \in \ran \chi_{[E_0 - \omega, E_0]}$ for some $E_0 \in \RR$ that vanishes to order $N$ at $0$.
 Eigenfunctions are definitely in some $\ran \chi_{[E_0 - \omega, E_0]}$ and by Theorem~\ref{thm:1}, we find an eigenfunction $f$, vanishing to order $N$ at $0$.
\end{proof}


\section*{Acknowledgments}
 The author thanks Tam\'as Korodi for bringing the sum-of-squares theorem to his attention.


\end{document}